\documentclass[11pt,twoside, leqno]{article}

\usepackage{amsthm}
\usepackage{amssymb}
\usepackage{amsmath}
\usepackage{mathrsfs}
\usepackage{txfonts}
\usepackage{graphics}
\usepackage{citeref}

\allowdisplaybreaks

\def\ls{\lesssim}

\def\fz{\infty}

\renewcommand{\r}{\right}
\newcommand{\lf}{\left}

\def\ls{\lesssim}

\def\supp{{\mathop\mathrm{\,supp\,}}}

\def\rr{{\mathbb R}}

\def\rn{{{\rr}^n}}

\def\nn{{\mathbb N}}

\newcommand{\wz}{\widetilde}

\def\epz{\varepsilon}

\def\fai{\varphi}
\def\gz{{\gamma}}

\def\wz{\widetilde}

\def\ls{\lesssim}

\def\boz{\Omega}

\def\divz{{{\mathop\mathrm {div}}}}

\newtheorem{theorem}{Theorem}[section]
\newtheorem{lemma}[theorem]{Lemma}
\newtheorem{corollary}[theorem]{Corollary}
\newtheorem{proposition}[theorem]{Proposition}
\theoremstyle{definition}
\newtheorem{remark}[theorem]{Remark}

\newtheorem{definition}[theorem]{Definition}
\def\supp{{\mathop\mathrm{\,supp\,}}}

\def\loc{{\mathop\mathrm{loc}}}

\numberwithin{equation}{section}

\begin{document}

\title{\Large\bf Some remarks on  Riesz transform on exterior Lipschitz domains
\footnotetext{\hspace{-0.35cm} 2020 {\it Mathematics Subject
Classification}. {Primary 35J25; Secondary 35B65, 42B35, 42B37.}
\endgraf{\it Key words and phrases}. Riesz transform, second-order elliptic operator, Dirichlet boundary condition, exterior Lipschitz domain.
\endgraf R. Jiang was partially supported by NNSF of China (12471094 \& 11922114). S. Yang is partially supported by the National Natural Science Foundation of China
(Grant No. 12431006), the Key Project of Gansu Provincial National
Science Foundation (Grant No. 23JRRA1022), the Fundamental Research Funds for the Central Universities (Grant No. lzujbky-2021-ey18) and the Innovative Groups of Basic Research in Gansu Province (Grant No. 22JR5RA391).}}

\author{Renjin Jiang and Sibei Yang}
\date{}
\maketitle

\vspace{-0.8cm}

\begin{center}
\begin{minipage}{12.5cm}\small
{{\bf Abstract.} Let $n\ge2$ and $\mathcal{L}=-\mathrm{div}(A\nabla\cdot)$
be an elliptic operator on $\mathbb{R}^n$. Given an exterior Lipschitz domain $\Omega$, let $\mathcal{L}_D$ be the elliptic operator $\mathcal{L}$
on $\Omega$ subject to the Dirichlet  boundary condition.
Previously it was known that the Riesz operator $\nabla \mathcal{L}_D^{-1/2}$
is not bounded for $p>2$ and $p\ge n$, even if $\mathcal{L}=-\Delta$ being the Laplace operator and $\Omega$ being a domain outside a ball.
Suppose that $A$ are CMO coefficients or VMO coefficients satisfying certain perturbation property, and $\partial\Omega$ is $C^1$, we prove that for $p>2$ and $p\in [n,\infty)$, it holds
$$
\inf_{\phi\in\mathcal{K}_p(\mathcal{L}_D^{1/2})}\left\|\nabla (f-\phi)\right\|_{L^p(\Omega)}\sim \inf_{\phi\in\mathcal{K}_p(\mathcal{L}_D^{1/2})}\left\|\mathcal{L}^{1/2}_D (f-\phi)\right\|_{L^p(\Omega)}
$$
for $f\in \dot{W}^{1,p}_0(\Omega)$. Here $\mathcal{K}_p(\mathcal{L}_D^{1/2})$ is the kernel of $\mathcal{L}_D^{1/2}$ in $\dot{W}^{1,p}_0(\Omega)$,
which coincides with $\tilde{\mathcal{A}}^p_0(\Omega):=\{f\in \dot{W}^{1,p}_0(\Omega):\,\mathcal{L}_Df=0\}$ and is a one dimensional subspace. 
As an application, we provide a substitution of  $L^p$-boundedness of $\sqrt{t}\nabla e^{-t\mathcal{L}_D}$ which is uniform in $t$ for $p\ge n$ and $p>2$. }
\end{minipage}
\end{center}

\vspace{0.1cm}

\section{Introduction and main results\label{s1}}

In this paper, motivated by the recent work \cite{hs09,jl22,kvz16} on the Riesz transform on exterior Lipschitz domains,
we continue to study the boundedness of the Riesz transform,
associated with second-order divergence form elliptic operators on the exterior Lipschitz domain $\boz$ having the
Dirichlet boundary condition, on $L^p(\boz)$ with $p\in(2,\fz)$.

Let $n\ge2$ and $\boz\subset\rn$ be an exterior Lipschitz domain, that is, the boundary of $\boz$, denoted by
$\partial\boz$, is a finite union of parts of rotated graphs of Lipschitz functions, $\boz$ is connected,
and $\boz^c:=\rn\setminus\boz$ is bounded. Furthermore, assume that $A\in L^\infty(\rn)$ is a real-valued
and symmetric matrix that satisfies the uniformly elliptic condition, that is,
there exists a constant $\mu_0\in(0,1]$ such that, for any $\xi\in\rn$ and $x\in\rn$,
\begin{equation*}
\mu_0|\xi|^2\le (A(x)\xi,\xi)\le\mu_0^{-1}|\xi|^2,
\end{equation*}
where $(\cdot,\cdot)$ denotes the inner product in $\rn$.

Denote by $\mathcal{L}$ the  operator $-\mathrm{div}(A\nabla\cdot)$ on
$\rn$, and by $\mathcal{L}_D$  the operator $-\mathrm{div}(A\nabla\cdot)$ on
$\boz$ subject to the Dirichlet boundary condition (see, for instance,
\cite[Section 4.1]{o05} for the detailed definitions of $\mathcal{L}$, $\mathcal{L}_D$).
When $A:=I_{n\times n}$ (the unit matrix), we simply denote these operators respectively by $\Delta$ and $\Delta_D$.
Moreover, let $O\subset\rn$ be a bounded Lipschitz domain. Denote by $\mathcal{L}_{D,O}$  the operator $-\mathrm{div}(A\nabla\cdot)$ on $O$ subject to the Dirichlet boundary condition.

Let $p\in[1,\fz]$. Denote by $\dot{W}^{1,p}(\rn)$ the \emph{homogeneous Sobolev space} on $\rn$, namely the completion
of $C^\fz_{\rm c}(\rn)$ under the \emph{semi-norm} $\|\nabla \cdot\|_{L^p(\rn)}$. Meanwhile, the \emph{homogeneous Sobolev spaces}
$\dot{W}_0^{1,p}(\boz)$, $\dot{W}^{1,p}(\boz)$, on $\boz$ is defined as the completion
of $C^\fz_{\rm c}(\boz)$, $C^\fz_{\rm c}(\rn)$, under the \emph{semi-norm} $\|\nabla \cdot\|_{L^p(\boz)}$ and $\|\nabla \cdot\|_{L^p(\rn)}$, respectively.
Here and thereafter, $C^\fz_{\rm c}(\rn)$ and $C^{\fz}_{\mathrm{c}}(\boz)$ denote the set of all \emph{infinitely differentiable
functions} with compact support in  $\rn$ and  $\boz$ respectively.
Moreover, for any $q\in(1,\fz)$, denote by $\dot{W}^{-1,q}(\rn)$, $\dot{W}^{-1,q}(\Omega)$ $\dot{W}_0^{-1,q}(\boz)$, respectively,
the dual spaces of $\dot{W}^{1,q'}(\rn)$, $\dot{W}^{1,q'}(\boz)$, $\dot{W}_0^{1,q'}(\boz)$,  where $q':=q/(q-1)$.

Let $O\subset\rn$ be a bounded Lipschitz domain. Denote by $W^{1,p}(O)$ the \emph{Sobolev space on $O$}
equipped with the \emph{norm}
$$\|f\|_{W^{1,p}(\boz)}:=\|f\|_{L^p(O)}+\|\,|\nabla f|\,\|_{L^p(O)},$$
where $\nabla f$ denotes the \emph{distributional gradient} of $f$. Furthermore, $W^{1,p}_{0}(O)$
is defined to be the \emph{closure} of $C^{\fz}_{\mathrm{c}}(O)$ in $W^{1,p}(O)$. Meanwhile,
for any $q\in(1,\fz)$, denote by $W^{-1,q}(O)$ and $W_0^{-1,q}(O)$, respectively,
the dual spaces of $W^{1,q'}(O)$ and $W_0^{1,q'}(O)$.

It is well known that the boundedness of the Riesz transform associated with some differential operators
on various function spaces has important applications in harmonic analysis and partial differential
equations and has aroused great interests in recent years (see, for instance,
\cite{acdh04,at01,cd99,cjks20,hs09,j21,jl20,kvz16,sh05a,sy10}).
In particular, let $O$ be a bounded Lipschitz domain of $\rn$. The sharp boundedness of the Riesz transform
$\nabla\mathcal{L}_{D,O}^{-1/2}$ associated with the operator $\mathcal{L}_{D,O}$ having the Dirichlet boundary condition
on the Lebesgue space $L^p(O)$ was established by Shen \cite{sh05a}.

Compared with the boundedness of the Riesz transform associated with differential operators on bounded Lipschitz domains,
there are relatively few literatures for the Riesz transform associated with differential operators on exterior Lipschitz domains.
Since the heat kernel generated by $\mathcal{L}_D$  satisfies the Gaussian upper bound estimate,
it follows from the results of Sikora \cite{s04} (see also \cite{cd99}) that the Riesz transforms
$\nabla\mathcal{L}^{-1/2}_D$ is always bounded on $L^p(\boz)$ for $p\in(1,2]$.
By studying weighted operators in the one dimension, Hassell and Sikora \cite{hs09} discovered that the Riesz
transform $\nabla\Delta_D^{-1/2}$ on the exterior of the unit ball is \emph{not} bounded on $L^p$ for $p\in(2,\fz)$ if $n=2$,
and $p\in[n,\fz)$ if $n\ge3$; see also \cite{lsz12} for the case $n=3$. Moreover, Killip, Visan and Zhang \cite{kvz16} proved that the Riesz
transform $\nabla\Delta_D^{-1/2}$ on the exterior of a smooth convex obstacle in $\rn$ ($n\ge3$)
is bounded for $p\in(1,n)$. Very recently,  characterizations for the boundedness
of the Riesz transform $\nabla\mathcal{L}_{D}^{-1/2}$ on $L^p(\boz)$ with $p\in(2,n)$ was obtained in \cite{jl22}.

Let
$$p({\mathcal{L}}):=\sup\{p>2: \,\nabla \mathcal{L}^{-1/2}\ \text{is bounded on}\, L^p(\rn)\}.$$
Furthermore, denote by $L^{1}_{\rm loc}(\mathbb{R}^n)$ the set of \emph{all locally integrable functions on}
${\mathbb R}^n$. Recall that the \emph{space} $\mathrm{BMO}(\rn)$ is defined as the set of all $f\in L^1_\loc(\rn)$
satisfying
$$
\|f\|_{\mathrm{BMO}(\rn)}:=\sup_{B\subset\rn}\frac{1}{|B|}\int_{B}\lf|f(x)-\frac{1}{|B|}\int_B f(y)\,dy\r|\,dx<\fz,
$$
where the supremum is taken over all balls $B$ of $\rn$ (see, for instance, \cite{jn61,St93}).
Moreover, the \emph{space} $\mathrm{CMO}(\rn)$ is defined as
the completion of $C^\fz_{\rm c}(\rn)$ in the space $\mathrm{BMO}(\rn)$ (see, for instance, \cite{cw77}).
The \emph{space} $\mathrm{VMO}(\rn)$ is defined as the set of $f\in \mathrm{BMO}(\rn)$ satisfying
$${\lim_{r\to 0}\sup_{x\in \rn}\frac{1}{|B(x,r)|}\int_{B(x,r)}\lf|f(y)-\frac{1}{|B(x,r)|}\int_{B(x,r)} f(z)\,dz\r|\,dy=0.}$$
Note that $\mathrm{CMO}(\rn)\varsubsetneq \mathrm{VMO}(\rn) \varsubsetneq\mathrm{BMO}(\rn).$
Let us recall some results proved in \cite[Theorem 1.3, Theorem 1.4]{jl22}.

\begin{theorem}[\cite{jl22}]\label{app-dirichlet}
 Let $\Omega\subset \rn$ be an exterior Lipschitz domain, $n\ge 2$.

 (i) For all $p\in (1,\infty)$, it holds for all $f\in \dot{W}^{1,p}_0(\Omega)$ that
\begin{equation}\label{e1.0}
\|\mathcal{L}_D^{1/2}f\|_{L^p(\Omega)}\le C\|\nabla f\|_{L^p(\Omega)}.
\end{equation}

(ii) Suppose that  $A\in \mathrm{VMO}(\rn)$ and $n\ge 3$. There exist  $\epsilon>0$ and  $C>1$  such that, for all $f\in \dot{W}^{1,p}_0(\Omega)$ it holds
\begin{equation}\label{e1.1}
C^{-1}\|\nabla f\|_{L^p(\Omega)}\le\lf\|\mathcal{L}_D^{1/2}f\r\|_{L^p(\Omega)}\le C\|\nabla f\|_{L^p(\Omega)},
\end{equation}
where $1<p<\min\{n,p({\mathcal{L}}),3+\epsilon\}$. If $\Omega$ is $C^1$, then { \eqref{e1.1}}
holds for all $1<p<\min\{n,p({\mathcal{L}})\}$.
\end{theorem}

\begin{remark}\rm
The version of \eqref{e1.0} for Neumann boundary operators $\mathcal{L}_N$ has been recently proved in \cite{dr22}
on complete manifolds with ends. Although the results in \cite{dr22} were presented in smooth manifolds setting, their proofs
extend to exterior Lipschitz domains almost identically and show that for all $p\in (1,\infty)$
\begin{equation*}
\|\mathcal{L}_N^{1/2}f\|_{L^p(\Omega)}\le C\|\nabla f\|_{L^p(\Omega)}.
\end{equation*}
Note that the heat kernel satisfies two side Gaussian bounds;
see \cite[Proof of Theorem 1.2]{jl22}.
\end{remark}

For the case $\mathcal{L}=-\Delta$ being the Laplacian operator and $\Omega$ being $C^1$, $p(\mathcal{L})=\infty$
and $\epsilon=\infty$. In this case, it follows from the above results that $\nabla\Delta_D^{-1/2}$ is bounded on $L^p(\Omega)$
for $1<p<n$.
By the unboundedness results on the Riesz transform $\nabla\Delta_D^{-1/2}$ established
in \cite{hs09}, the range $(1,\min\{n,3+\epz\})$ of $p$ for \eqref{e1.1} is sharp; see also \cite{jl22,kvz16}.

The main purpose of this paper is to further investigate the case $p\ge n$.
Note that from Theorem \ref{app-dirichlet}, the boundedness of the Riesz transform $\nabla \mathcal{L}_D^{-1/2}$
depends on $n,p(\mathcal{L})$ and the geometry of the boundary $\partial\Omega$. All the dependence are essential, see the characterizations obtained by \cite[Theorem 1.1]{jl22},
the regularity dependence of the boundary by \cite{jk95}, and the counter-examples provided in \cite{hs09,kvz16,jl22}.
However, for operator with nice coefficients and domain with nice boundary ($C^1$ or small Lipschitz constant)
such that $p(\mathcal{L}),3+\epsilon\ge n$, we can find a suitable substitution of $\dot{W}^{1,p}_0(\Omega)$ space for
the inequality \eqref{e1.1} as following.

Let us assume that the matrix $A$ in the operator $\mathcal{L}$ is in the space $\mathrm{VMO}(\rn)$ and satisfies the perturbation
$$\fint_{B(x_0,r)}|A-I_{n\times n}|\,dx\le \frac{C}{r^\delta} \leqno(GD)$$
for some $\delta>0$, all $r>1$ and all $x_0\in\rn$.
Or we assume that $A\in \mathrm{CMO}(\rn)$. In both cases, from \cite{jl20} and  \cite[Theorem 1]{is98} respectively,
it is known that
$$p(\mathcal{L})=\infty.$$

We have the following replacement for the Riesz inequality for $p\ge n$ and $p>2$.
\begin{theorem}\label{c1.2}
Let $n\ge2$ and $\boz\subset\rn$ be an exterior $C^1$ domain. Assume that { $A\in\mathrm{VMO}(\rn)$ satisfies $(GD)$ or $A\in\mathrm{CMO}(\rn)$}.
Let $p>2$ and $p\in [n,\infty)$.

(i) The kernel space $\mathcal{K}_{p}(\mathcal{L}_D^{1/2})$ of $\mathcal{L}_D^{1/2}$ in $\dot{W}^{1,p}_0(\Omega)$ coincides with $\tilde{\mathcal{A}}^p_0(\boz):=\{\phi\in \dot{W}_0^{1,p}(\boz):\ \mathcal{L}_D f=0\}.$ Moreover, when $n\ge3$, $\tilde{\mathcal{A}}^p_0(\boz)=\mathcal{A}^p_0(\boz):=\{c(u_0-1): \ c\in\rr\},$
where $u_0$ is the unique solution in $\dot{W}^{1,2}(\boz)\cap\dot{W}^{1,p}(\boz)$ of the problem
\begin{equation*}
\lf\{\begin{array}{ll}
-\mathrm{div}(A\nabla u_0)=0\ \ &\text{in}\ \boz,\\
u_0=1\  &\text{on}\  \partial\boz;
\end{array}\r.
\end{equation*}
when $n=2$, $\tilde{\mathcal{A}}^p_0(\boz)=\mathcal{A}^p_0(\boz):=\{c(u_0-u_1): \ c\in\rr\},$
where $u_0$ is the unique solution in $\dot{W}^{1,2}(\boz)\cap\dot{W}^{1,p}(\boz)$ of the problem
\begin{equation*}
\lf\{\begin{array}{ll}
-\mathrm{div}(A\nabla u_0)=0\ \ &\text{in}\ \boz,\\
u_0=u_1\  &\text{on}\  \partial\boz,
\end{array}\r.
\end{equation*}
and $u_1\in\dot{W}^{1,p}(\rr^2)$ is a solution of the problem
$\mathcal{L}u=\frac{1}{|\partial\boz|}\mathbf{1}_{\partial\boz}$ in $\rr^2$. Here $\mathbf{1}_{\partial\boz}$ denotes
the characteristic function of $\partial\boz$ in $\rr^2$.

(ii) It holds for all
$f\in \dot{W}_0^{1,p}(\boz)$ that
\begin{equation}\label{e1.3}
\inf_{\phi\in \mathcal{K}_p(\mathcal{L}_D^{1/2})}\lf\|\nabla f-\nabla\phi\r\|_{L^p(\boz)}\le C\lf\|\mathcal{L}^{1/2}_D f\r\|_{L^p(\boz)},
\end{equation}
and consequently, it holds that
\begin{equation*}
\inf_{\phi\in \mathcal{K}_p(\mathcal{L}_D^{1/2})}\lf\|\nabla (f-\phi)\r\|_{L^p(\boz)}\sim  \inf_{\phi\in \mathcal{K}_p(\mathcal{L}_D^{1/2})}\lf\|\mathcal{L}^{1/2}_D (f-\phi)\r\|_{L^p(\boz)}.
\end{equation*}
\end{theorem}
The \emph{symbol} $f\sim g$ means
$f\ls g$ and $g\ls f$, which stands for $f\le Cg$ and $g\le Ch$.
The main new ingredient appeared in Theorem \ref{c1.2} is identifying the kernel $\mathcal{K}_{p}(\mathcal{L}_D^{1/2})$ of $\mathcal{L}_D^{1/2}$ in $\dot{W}^{1,p}_0(\Omega)$
as the space $\mathcal{A}^p_0(\boz)$, which is motivated by the work of Amrouche, Girault and Giroire \cite{agg97}.
We can actually establish a more general version of Theorem \ref{c1.2}, provided that $p(\mathcal{L})\ge n$
and the boundary $\partial\Omega$ is $C^1$ or with small Lipschitz constant; see Theorem \ref{t1.1} below.

Let us remark that we can have an explicit description in the exterior setting is due to the boundedness of the Riesz transform in $\rn$
for $1<p<\infty$ and the special geometry of exterior domains. From previous results  of Riesz transforms from \cite{acdh04,cjks20,jl20}, we know in case of $p\in (2,\infty)$, both locally and global  geometry can destroy
the boundedness of the Riesz transform. In particular, a local perturbation of $A$ may result in huge difference
of behavior of the Riesz transform for $p>2$; see \cite{jl20} for instance. So generally speaking it is hard (at least to us) to have an explicit
description of the kernel space. For the case of exterior domains, under the assumption of $p(\mathcal{L})=\infty$, we see that the kernel space that breaks down the boundedness of the
Riesz transform for $p\ge n$ and $p>2$  is actually only {\em one dimensional subspace} of $\dot{W}^{1,p}(\boz)$.

Finally let us apply Theorem \ref{c1.2} to the mapping property of the gradient of heat semigroup, which plays important roles in the study of
of Schr\"odinger equations, see \cite{I10,IP08,IP10,kvz16b,lsz12} for instance. For the operator
$\sqrt t\nabla e^{-t\mathcal{L}_\Omega}$, it was known that there are no uniform $L^p$-bounds in $t$ for $p>n$; see \cite[Proposition 8.1]{kvz16}.
As an application of \eqref{e1.3} of Theorem \ref{c1.2} , we have the following substitution.

\begin{theorem}
Let $n\ge2$ and $\boz\subset\rn$ be an exterior $C^1$ domain. Assume that { $A\in\mathrm{VMO}(\rn)$ satisfies $(GD)$ or $A\in\mathrm{CMO}(\rn)$}.
Let $p>2$ and $p\in [n,\infty)$. Then it holds that
\begin{equation*}
\inf_{\phi\in \mathcal{K}_p(\mathcal{L}_D^{1/2})} \lf\|\sqrt t \nabla e^{-t\mathcal{L}_D}f-\nabla\phi\r\|_{L^p(\boz)}\le C\|f\|_{L^p(\boz)},\,\forall\,t>0.
\end{equation*}
\end{theorem}
The proof is straightforward by using \eqref{e1.3} and the analyticity of the heat semigroup, as
\begin{equation*}
\inf_{\phi\in \mathcal{K}_p(\mathcal{L}_D^{1/2})} \lf\|\sqrt t \nabla e^{-t\mathcal{L}_D}f-\nabla\phi\r\|_{L^p(\boz)}\le C\lf\|\sqrt t \mathcal{L}_D^{1/2} e^{-t\mathcal{L}_D}f\r\|_{L^p(\boz)}\le C\|f\|_{L^p(\boz)},\,\forall\,t>0.
\end{equation*}

In the particular case $\mathcal{L}=\Delta$ and $\Omega=\rn\setminus B(0,1)$, it is clear that the kernel space is exactly as
\begin{align}
\mathcal{K}_{p}(\Delta_D^{1/2})=\tilde{\mathcal{A}}^p_0(\boz)=
\begin{cases}
\{c(1-|x|^{2-n}),\, |x|>1,\,c\in\rr\}, & n\ge 3,\\
\{c\log|x|,\,|x|>1,\, c\in\rr\}, & n=2,
\end{cases}
\end{align}
where $p\ge n$ and $p>2$.  We therefore have the following corollary.
\begin{corollary}\label{c-heat-ball}
Let $n\ge2$ and $\Omega=\rn\setminus B(0,1)$.
Let $p>2$ and $p\in [n,\infty)$. Then it holds that
\begin{equation*}
\inf_{c\in\rr} \lf\|\sqrt t \nabla e^{-t\Delta_D}f-\frac{cx}{|x|^n}\r\|_{L^p(\boz)}\le C\|f\|_{L^p(\boz)},\,\forall\,t>0.
\end{equation*}
\end{corollary}
It is clear from \cite[Proposition 8.1]{kvz16} that in the LHS of the last inequality the  infimum is not attained at $c=0$. Moreover, since for  $f\in L^p(\Omega)$,
$\sqrt t \nabla e^{-t\mathcal{L}_\Omega}f$ does belong to $L^p(\Omega)$ (without uniform bound in $t$), the infimum shall be attained at the finite $c$ which depends on $f$ and $t$.

We shall first prove an intermediate version of Theorem \ref{c1.2} in Section 2, we shall then show the equivalence of the spaces $\mathcal{A}^p_0(\boz)$, $\tilde{\mathcal{A}}^p_0(\boz)$ and
$K_p(\mathcal{L}_D^{1/2})$ and complete the proof of Theorem \ref{c1.2} in Section 3.

Throughout the whole {paper}, we always denote by $C$ or
$c$ a \emph{positive constant} which is independent of the main parameters, but it may vary from line
to line.
For any measurable subset $E$ of $\rn$, we denote by $E^c$
the \emph{set} $\rn\setminus E$ and by $\mathbf{1}_{E}$ its \emph{characteristic function}. Furthermore, for any $q\in[1,\fz]$,
we denote by $q'$ its \emph{conjugate exponent}, namely $1/q+1/q'= 1$. Finally, for any measurable set
$E\subset\rn$ and (vector-valued or matrix-valued) function $f\in L^1(E)$, we denote the integral
$\int_{E}|f(x)|\,dx$ simply by $\int_{E}|f|\,dx$ and, when $|E|<\fz$, we use the notation
$$
(f)_E:=\fint_Ef(x)\,dx:=\frac{1}{|E|}\int_{E}f(x)\,dx.
$$

\section{On boundedness of the Riesz transform}\label{s2}
In this section we prove the following more general version Theorem \ref{t1.1} of Theorem \ref{c1.2} (ii) with $\mathcal{K}_p(\mathcal{L}_D^{1/2})$ replaced by $\tilde{\mathcal{A}}^p_0(\boz)$, which is defined as
$$\tilde{\mathcal{A}}^p_0(\boz)=\lf\{w\in\dot{W}_0^{1,p}(\boz):\ \mathcal{L}_Dw=0\r\}.$$
Let us begin with some necessary notations.

\begin{definition}\label{d1.1}
Let $\mathcal{L}:=-\mathrm{div}(A\nabla\cdot)$ be a second-order divergence form elliptic operator on $\rn$.
Denote by $(q(\mathcal{L})', q(\mathcal{L}))$ the \emph{interior of the maximal interval of exponents}
$q\in[1,\infty]$ such that the operator $\nabla\mathcal{L}^{-1}\mathrm{div}$ is bounded on $L^q(\rn)$.

Furthermore, let $O$ be a bounded Lipschitz domain of $\rn$ and let $\mathcal{L}_{D,O}:=-\mathrm{div}(A\nabla\cdot)$
be a second-order divergence form elliptic operator on $O$ subject to the Dirichlet boundary condition. Similarly,
denote by $(q(\mathcal{L}_{D,O})', q(\mathcal{L}_{D,O}))$ the \emph{interior of the maximal interval of exponents}
$q\in[1,\infty]$ such that $\nabla\mathcal{L}_{D,O}^{-1}\mathrm{div}$ is bounded on $L^q(O)$.
\end{definition}

\begin{remark}\label{r1.1}
It is well known that there exists a constant $\epz_0\in(0,\fz)$ depending on the matrix $A$ and $n$ such
that $(2-\epz_0,2+\epz_0)\subset(q(\mathcal{L})', q(\mathcal{L}))$ (see, for instance, \cite{is98}).
Similarly, there exists a constant $\epz_1\in(0,\fz)$
depending on $A$, $n$, and the Lipschitz constant of $O$ such that $(2-\epz_1,2+\epz_1)\subset(q(\mathcal{L}_{D,O})',
q(\mathcal{L}_{D,O}))$.
\end{remark}

\begin{remark}
Note that $q(\mathcal{L})=p(\mathcal{L})$. In fact, since for $1<p<\infty$  it holds that
\begin{eqnarray*}
\lf\|\mathcal{L}^{1/2}f\r\|_{L^p(\rn)}\le C\|\nabla f\|_{L^p(\rn)}
\end{eqnarray*}
(see \cite{at01}), one further has
\begin{eqnarray*}
\lf\|\mathcal{L}^{-1/2}\mathrm{div}\r\|_{p\to p}= \lf\|\mathcal{L}^{1/2}\mathcal{L}^{-1}\mathrm{div}\r\|_{p\to p}\le
 C\lf\|\nabla \mathcal{L}^{-1}\mathrm{div}\r\|_{p\to p},
\end{eqnarray*}
which by duality implies that, the $L^p$-boundeness of $\nabla \mathcal{L}^{-1}\mathrm{div}$ implies $L^{p'}$-boundedness of $\nabla \mathcal{L}^{-1/2}.$
On the other hand, note that for $p\in (1,p(\mathcal{L}))$, $\nabla \mathcal{L}^{-1/2}$ is bounded on $L^p(\rn)$. Therefore, for $p\in (p(\mathcal{L})',p(\mathcal{L}))$, we have that
\begin{eqnarray*}
\lf\|\nabla \mathcal{L}^{-1}\mathrm{div}\r\|_{p\to p}=\lf\|\nabla \mathcal{L}^{-1/2}\lf(\mathcal{L}^{-1/2}\mathrm{div}\r)\r\|_{p\to p}\le\lf\| \nabla \mathcal{L}^{-1/2}\r\|_{p\to p} \lf\|\nabla \mathcal{L}^{-1/2}\r\|_{p'\to p'}<\infty.
\end{eqnarray*}
Thus we have  $q(\mathcal{L})=p(\mathcal{L})$.
\end{remark}
In what follows, for any $x\in\rn$ and $r\in(0,\fz)$, we always let $B(x,r):=\{y\in\rn:\ |y-x|<r\}$. Note that on $\rn$ the maximal interval
for the $L^p$-boundedness of the Riesz transform is open (cf. \cite{cjks20}), so we may assume that   $p(\mathcal{L})=q(\mathcal{L})>n$.
\begin{theorem}\label{t1.1}
Let $n\ge2$ and $\boz\subset\rn$ be an exterior Lipschitz domain. Take a large $R\in(0,\fz)$ such that $\boz^c\subset B(0,R)$.
Let $\boz_R:=\boz\cap B(0,R)$. { Assume that $\min\{q(\mathcal{L}),q(\mathcal{L}_{D,\boz_R})\}>n$ and $2<p\in [n, \min\{q(\mathcal{L}),
q(\mathcal{L}_{D,\boz_R})\})$.} Then there exists a
positive constant $C$ such that, for any $f\in \dot{W}_0^{1,p}(\boz)$,
\begin{equation*}
\inf_{\phi\in\tilde{\mathcal{A}}^p_0(\boz)}\lf\|\nabla f-\nabla\phi\r\|_{L^p(\boz)}\le C\lf\|\mathcal{L}^{1/2}_D f\r\|_{L^p(\boz)},
\end{equation*}
where $\tilde{\mathcal{A}}^p_0(\boz)=\{\phi\in \dot{W}_0^{1,p}(\boz):\ \mathcal{L}_D f=0\}$.
\end{theorem}

%

To prove Theorem \ref{t1.1}, let us first begin with the following several lemmas.

Let $X$ be a Banach space and $Y$ a closed subspace of $X$. Denote by $X^\ast$ the \emph{dual space} of $X$.
Let
$$X^\ast\bot Y:=\{f\in X^\ast:\ \text{for\ all}\ x\in Y,\ \langle f,x\rangle\}=0,
$$
where $\langle\cdot,\cdot\rangle$ denotes the duality pairing between $X^\ast$ and $X$. That is,
$X^\ast\bot Y$ denotes the subspace of $X^\ast$ orthogonal to $Y$.

Meanwhile, for any given $m\in\nn\cup\{0\}$, we denote by $\mathcal{P}_m$ the space of polynomials on $\rn$
of degree less than or equal to $m$; if $m$ is a strictly negative integer, we set by convention $\mathcal{P}_m=\{0\}$.
Moreover, for any $s\in\rr$, denote by $\lfloor s\rfloor$ the maximal integer not more than $s$.

Then we have the following conclusion on the isomorphism property of the divergence operator $\mathrm{div}$
which was essentially obtained in \cite[Propositions 4.1 and 9.2]{agg94}.

\begin{lemma}\label{l2.0}
Let $n\ge2$, $p\in(1,\fz)$, and $p'\in(1,\fz)$ be given by $1/p+1/p'=1$.
Then the divergence operator $\mathrm{div}$ is an isomorphism from $L^{p}(\rn)/H_{p}$ to $\dot{W}^{-1,p}(\rn)\bot
\mathcal{P}_{\lfloor 1-n/p'\rfloor}$, where $H_{p}:=\{v\in L^{p}(\rn):\ \mathrm{div}(v)=0\ \text{in the sense of distributions}\}$.
\end{lemma}

\begin{lemma}\label{l2.1}
Let $p\in(2,\fz)$ and $f\in \dot{W}^{-1,p}(\rn)$ with compact support.
Then $f\in \dot{W}^{-1,2}(\rn)$ and there exists a positive constant $C$, depending only on $p$
and the support of $f$, such that
$$
\|f\|_{\dot{W}^{-1,2}(\rn)}\le C\|f\|_{\dot{W}^{-1,p}(\rn)}.
$$
\end{lemma}
\begin{proof}
Suppose that $\supp f\subset B(0,R)$ and take a bump function $\psi_R$ such that $\psi_R=1$ on $B(0,R)$,
$\supp\psi_R\subset B(0,R+1)$ and $|\nabla \psi_R|\le 1$.
For any $g\in C^\infty_{\rm c}(\rn)$, one has
{ \begin{align*}
|\langle f,g\rangle|&=|\langle f,g\psi_R\rangle|
\le \|f\|_{\dot{W}^{-1,p}(\rn)} \|\nabla (g\psi_R)\|_{L^{p'}(\rn)}\\
&\le \|f\|_{\dot{W}^{-1,p}(\rn)}\left(\|\nabla g\|_{L^{p'}(B(0,R+1))}+ \|g\nabla\psi_R\|_{L^{p'}(\rn)}\right)\\
&\le  C(R)\|f\|_{\dot{W}^{-1,p}(\rn)} \left(\|\nabla g\|_{L^{2}(\rn)}+ \left(\int_{\rn}\frac{|g(x)|^2}{|x|^2} \,dx\right)^{1/2}\right)\\
&\le C(R)\|f\|_{\dot{W}^{-1,p}(\rn)}\|\nabla g\|_{L^{2}(\rn)},
\end{align*}}
where the last inequality follows from the Hardy inequality; see \cite{ckn84}.
The proof is completed by taking supremum over $g$ w.r.t.  the norm $\|\nabla g\|_{L^{2}(\rn)}$.
\end{proof}
\begin{lemma}\label{l2.2}
Let $n\ge2$ and $p\in(2,q(\mathcal{L}))$. Assume that $f\in \dot{W}^{-1,p}(\rn)$ has compact support.
When $n\ge3$, the problem
\begin{equation}\label{e2.1}
\mathcal{L}u=f\ \text{in}\ \rn
\end{equation}
has a unique solution $u$ in $\dot{W}^{1,2}(\rn)\cap\dot{W}^{1,p}(\rn)$  up to constants.

When $n=2$, if $f$ further satisfies the compatibility condition $\langle f,1\rangle=0$, then the problem \eqref{e2.1} has a unique solution
$u$ in $\dot{W}^{1,2}(\rr^2)\cap\dot{W}^{1,p}(\rr^2)$ up to constants.
\end{lemma}

\begin{proof}
By the assumption that $f\in \dot{W}^{-1,p}(\rn)$ has compact support and Lemma \ref{l2.1},
we conclude that $f\in \dot{W}^{-1,2}(\rn)$.

When $n\ge3$, from Lemma \ref{l2.0}, the definition of the interval $(q_-(\mathcal{L}),q(\mathcal{L}))$,
and the assumption $p\in(2,q(\mathcal{L}))$, it follows that the equation \eqref{e2.1} has a unique solution
$u\in \dot{W}^{1,2}(\rn)\cap\dot{W}^{1,p}(\rn)$ up to constants.

Moreover, when $n=2$, by the known fact that constants belong to $\dot{W}^{1,2}(\rr^2)$,
we find that the duality pairing $\langle f,1\rangle$ makes sense. Therefore, in this case,
using Lemma \ref{l2.0} and  the assumption $p\in(2,q(\mathcal{L}))$ again, we conclude that
the problem \eqref{e2.1} has a unique solution
$u\in \dot{W}^{1,2}(\rr^2)\cap\dot{W}^{1,p}(\rr^2)$ up to constants. This finishes the proof of Lemma \ref{l2.2}.
\end{proof}

\begin{lemma}\label{l2.3}
Let $n\ge2$ and $\boz\subset\rn$ be an exterior Lipschitz domain. Take a large $R\in(0,\fz)$ such that $\boz^c\subset
B(0,R)$ and let $\boz_R:=\boz\cap B(0,R)$. Let $p\in(2, \min\{q(\mathcal{L}),q(\mathcal{L}_{D,\boz_R})\})$.
Assume that $f\in \dot{W}_0^{-1,p}(\boz)$ has compact support and its support is contained in $B(0,R)$.
Then the Dirichlet problem
\begin{equation}\label{e2.2}
\lf\{\begin{array}{ll}
-\divz(A\nabla u)=f\ \ &\text{in}\ \boz,\\
u=0\  &\text{on}\  \partial\boz
\end{array}\r.
\end{equation}
has a unique solution $u$ in $\dot{W}_0^{1,2}(\boz)\cap\dot{W}_0^{1,p}(\boz)$.
\end{lemma}

Let $s\in(0,1)$ and $p\in(1,\fz)$. For the exterior Lipschitz domain (or the bounded Lipschitz domain) $\boz$ of $\rn$,
denote by $W^{s,p}(\partial \boz)$ the \emph{fractional Sobolev space on $\partial \boz$} (see, for instance, \cite[Section 2.4.3]{j13}
for its definition). To show Lemma \ref{l2.3}, we need the following conclusion.

\begin{lemma}\label{l2.4}
Let $n\ge2$ and $O\subset\rn$ be a bounded Lipschitz domain. Let $p\in(2, q(\mathcal{L}_{D,O}))$.
Assume that $f\in \dot{W}_0^{-1,p}(O)$ and $g\in W^{1/p',p}(\partial O)$.
Then the Dirichlet problem
\begin{equation}\label{e2.3}
\lf\{\begin{array}{ll}
-\divz(A\nabla v)=f\ \ &\text{in}\ O,\\
v=g\  &\text{on}\  \partial O
\end{array}\r.
\end{equation}
has a unique solution $v$ in $W^{1,p}(O)$.
\end{lemma}

\begin{proof}
We first prove that there exists a solution $v\in W^{1,p}(O)$ for the problem \eqref{e2.3}.
Indeed, by $g\in W^{1/p',p}(\partial O)$ and the converse trace theorem for Sobolev spaces
(see, for instance, \cite[Section 2.5.7, Theorem 5.7]{j13}), we find that there exists
a function $w_1\in W^{1,p}(O)$ such that $w_1=g$ on $\partial O$. Moreover, it is easy
to find that $-\mathrm{div}(A\nabla w_1)\in W^{-1,p}_0(O)$. Furthermore, from the assumption that
$p\in(2, q(\mathcal{L}_{D,O}))$, it follows that there exists a unique $w_2\in W^{1,p}_0(O)$
satisfying
\begin{equation*}
\lf\{\begin{array}{ll}
-\mathrm{div}(A\nabla w_2)=f+\mathrm{div}(A\nabla w_1)\ \ &\text{in}\ O,\\
w_2=0\  &\text{on}\  \partial O.
\end{array}\r.
\end{equation*}
Thus, $v:=w_1+w_2\in W^{1,p}(O)$ is a solution of the problem \eqref{e2.3}.

Now, we show the solution of \eqref{e2.3} is unique. Assume that $v_1,v_2\in W^{1,p}(O)$ are solutions of \eqref{e2.3}.
Then $\mathrm{div}(A\nabla (v_1-v_2))=0$ in $O$ and $v_1-v_2=0$ on $\partial O$.
Thus, $v_1=v_2$ almost everywhere in $O$.
This finishes the proof of Lemma \ref{l2.4}.
\end{proof}

Now, we prove Lemma \ref{l2.3} by using Lemmas \ref{l2.2} and \ref{l2.4}.

\begin{proof}[Proof of Lemma \ref{l2.3}]
By the assumption that $f\in \dot{W}_0^{-1,p}(\boz)$ has compact support and an argument similar to the proof of
Lemma \ref{l2.1}, we conclude that $f\in \dot{W}_0^{-1,2}(\boz)$, which, combined with the Lax--Milgram theorem,
further implies that the Dirichlet problem \eqref{e2.2} has a unique solution $u\in \dot{W}_0^{1,2}(\boz)$.

Next, we show $u\in \dot{W}_0^{1,p}(\boz)$. We first assume that
$$p\in\lf(2,\min\lf\{q(\mathcal{L}),q(\mathcal{L}_{D,\boz_R}),\frac{2n}{n-2}\r\}\r)$$
when $n\ge3$ or $p\in(2,\min\{q(\mathcal{L}),q(\mathcal{L}_{D,\boz_R})\})$ when $n=2$.

Let $\fai_1,\fai_2\in C^\fz(\rn)$ satisfy $0\le\fai_1,\fai_2\le1$, $\supp(\fai_1)\subset B(0,R+1)$, $\fai_1\equiv1$ on $B(0,R)$,
and $\fai_1+\fai_2\equiv1$ in $\rn$. Extend $u$ by zero in $\boz^c$ and let $u=u_1+u_2$,
where $u_1:=u\fai_1$ and $u_2:=u\fai_2$. Then
$$
\mathrm{div}(A\nabla u_2)=\mathrm{div}(A\nabla(u\fai_2))\ \text{in}\ \rn.
$$
From $u\in\dot{W}_0^{1,2}(\boz)$ and the assumptions that $\fai_2\in C^\fz(\rn)$ and $\fai_2\equiv1$ on $\rn\backslash B(0,R+1)$,
we infer that $u\fai_2\in W^{1,2}(\boz\cap B(0,R+1))$ and hence { $\nabla(u\fai_2)\in L^2(\rn)$}, which, together with the assumption
$A\in L^\fz(\rn;\rr^{n\times n})$, further implies that { $A\nabla(u\fai_2)\in L^2(\rn)$}.

Furthermore, it is straight to see that
$$
\mathrm{div}(A\nabla(u\fai_2))=f\fai_2-A\nabla u\cdot\nabla\fai_2-\mathrm{div}(uA\nabla\fai_2):=g
$$
in the weak sense. By the assumption $f\in\dot{W}_0^{-1,p}(\boz)$, we conclude that $f\fai_2\in\dot{W}_0^{-1,p}(\boz)$.
Meanwhile, from $u\in\dot{W}_0^{1,2}(\boz)$ and the assumptions that $\fai_2\in C^\fz(\rn)$, $0\le\fai_2\le1$, and $\fai_2\equiv1$ on
$\rn\backslash B(0,R+1)$, we deduce that $A\nabla u\cdot\nabla\fai_2\in L^2(\boz\cap B(0,R+1))$,
which, combined with the Sobolev inequality, further implies that $A\nabla u\cdot\nabla\fai_2\in W^{-1,p}(B(0,R+1))$. Moreover,
by $u\in\dot{W}^{1,2}(\boz)$, we find that $u\in L^{p}_\loc(\boz)$, which, together with the assumptions that $\fai_2\in C^\fz(\rn)$
and $\supp(\nabla\fai_2)\subset B(0,R+1)$, further implies that { $Au\nabla\fai_2\in L^{p}(\boz\cap B(0,R+1))$} and hence
$\mathrm{div}(Au\nabla\fai_2)\in W^{-1,p}_0(\boz\cap B(0,R+1))$. Thus, we have $g\in W^{-1,p}_0(\boz\cap B(0,R+1))$.
Extend $g$ by zero in $\boz^c$. Then $g\in \dot{W}^{-1,p}(\rn)$. Therefore,
$$
-\mathrm{div}(A\nabla u_2)=g\ \text{in} \ \rn,
$$
which, combined with Lemma \ref{l2.0} and $p\in(2,q(\mathcal{L}))$, further implies that
$u_2\in\dot{W}^{1,p}(\rn)$.

Furthermore, from $u=u_2$ on $\partial B(0,R+1)$ and the trace theorem for Sobolev spaces (see, for instance,
\cite[Section 2.5.4, Theorem 5.5]{j13}), it follows that $u\in W^{1/p',p}(\partial B(0,R+1))$.
Meanwhile, we have
\begin{equation}\label{e2.4}
\lf\{\begin{array}{ll}
-\mathrm{div}(A\nabla u)=f\ \ &\text{in}\ \boz_{R+1},\\
u=0\  &\text{on}\  \partial\boz,\\
u=u_2\  &\text{on}\  \partial B(0,R+1),\\
\end{array}\r.
\end{equation}
where $\boz_{R+1}:=\boz\cap B(0,R+1)$.
By the assumption $p<q(\mathcal{L}_{D,\boz_{R+1}})$ and Lemma \ref{l2.4}, we conclude that the problem
\eqref{e2.4} has a unique solution in $W^{1,p}(\boz_{R+1})$, which further implies that $u\in W^{1,p}(\boz_{R+1})$.
From this, $u\in \dot{W}^{1,2}_0(\boz)$, $u_2\in\dot{W}^{1,p}(\rn)$, and the fact that $u=u_2$ on
$\rn\backslash B(0,R+1)$, we deduce that $u\in\dot{W}^{1,p}(\boz)$ with any given
$$p\in\lf(2,\min\lf\{q(\mathcal{L}),q(\mathcal{L}_{D,\boz_{R+1}}),\frac{2n}{n-2}\r\}\r)$$
when $n\ge3$ or any given $p\in(2,\min\{q(\mathcal{L}),q(\mathcal{L}_{D,\boz_{R+1}})\})$
when $n=2$. Then, using an bootstrap argument (see, for instance, \cite[p.\,63]{agg97}), we find that $u\in
\dot{W}^{1,p}_0(\boz)$ with any given $p\in(2,\min\{q(\mathcal{L}),q(\mathcal{L}_{D,\boz_{R+1}})\})$.
This finishes the proof of Lemma \ref{l2.3}.
\end{proof}

\begin{lemma}\label{l2.5}
Let $n\ge2$ and $\boz\subset\rn$ be an exterior Lipschitz domain. {Take a large $R\in(0,\fz)$ such that $\boz^c\subset B(0,R)$ and let $\boz_R:=\boz\cap B(0,R)$.
Assume that $\min\{q(\mathcal{L}),q(\mathcal{L}_{D,\boz_R})\}>n$. Let $p>2$ and $p\in [n,\min\{q(\mathcal{L}),q(\mathcal{L}_{D,\boz_R})\})$.
Assume further that $f\in \dot{W}_0^{-1,p}(\boz)$ and}
$$\tilde{\mathcal{A}}^p_0(\boz)=\lf\{w\in\dot{W}_0^{1,p}(\boz):\ \mathcal{L}_Dw=0\r\}.$$
Then the problem
\begin{equation}\label{e2.5}
\lf\{\begin{array}{ll}
-\divz(A\nabla u)=f\ \ &\text{in}\ \boz,\\
u=0\  &\text{on}\  \partial\boz
\end{array}\r.
\end{equation}
has a unique solution $u$ in $\dot{W}_0^{1,p}(\boz)/\tilde{\mathcal{A}}^p_0(\boz)$ and there exists a positive constant independent of $f$ such that
\begin{equation*}
\inf_{\phi\in \tilde{\mathcal{A}}^p_0(\boz)}\lf\|\nabla u-\nabla\phi\r\|_{L^p(\boz)}\le C\|f\|_{\dot{W}_0^{-1,p}(\boz)}.
\end{equation*}
\end{lemma}
\begin{remark}
Note that the above lemma is non-trivial only if $\min\{q(\mathcal{L}),q(\mathcal{L}_{D,\boz_R})\}\ge n$ and $\min\{q(\mathcal{L}),q(\mathcal{L}_{D,\boz_R})\}>2$.
 This is not surprise,
since by \cite[Theorem 1.1]{jl22} and a similar proof of \cite[Theorem 1.4]{jl22} { via using the role of} $q(\mathcal{L}_{D,\boz_R})$ instead of
using  \cite[Theorem B \& Theorem C]{sh05a} there, the Riesz operator $\nabla \mathcal{L}^{-1/2}$ is bounded for $p\in (1,n)\cup(1,2]$.
In this case the kernel $\mathcal{A}^p_0(\boz)$ must be trivial, i.e., equal zero.
\end{remark}

\begin{proof}[Proof of Lemma \ref{l2.5}]
By the Closed Range Theorem of Banach (see, for instance, \cite[Theorem 5.11-5]{c13}), we find that there exists
a vector-valued function $F\in L^p(\boz)$
such that $f=\mathrm{div}F$ in $\boz$. Extend $F$ by zero in $\boz^c$ and still denote this extension by $F$.
Let $\wz{f}:=\mathrm{div}F$. Then $\wz{f}\in \dot{W}^{-1,p}(\rn)$. From Lemma \ref{l2.0}, the assumption that $p\in(2,q(\mathcal{L}))$,
and the definition of the interval $(q(\mathcal{L}))',q(\mathcal{L}))$, it follows that there exists a
unique $w\in\dot{W}^{1,p}(\rn)$ up to constants such that
\begin{equation*}
\mathcal{L}w=\wz{f}\ \text{in}\ \rn.
\end{equation*}
Moreover, consider the Dirichlet problem
\begin{equation}\label{e2.6}
\lf\{\begin{array}{ll}
-\mathrm{div}(A\nabla z)=0\ \ &\text{in}\ \boz,\\
z=-w\  &\text{on}\  \partial\boz.
\end{array}\r.
\end{equation}
Then the problem \eqref{e2.6} has a unique solution $z\in\dot{W}^{1,2}(\boz)\cap\dot{W}^{1,p}(\boz)$.
Indeed, take a large $R\in(0,\fz)$ such that $\boz^c\subset B(0,R)$ and let $\boz_R:=\boz\cap B(0,R)$.
By $w\in\dot{W}^{1,p}(\rn)$, we conclude that $w\in W^{1/p',p}(\partial\boz)$. Let $u_z$ satisfy
\begin{equation}\label{e2.7}
\lf\{\begin{array}{ll}
-\mathrm{div}(A\nabla u_z)=0\ \ &\text{in}\ \boz_R,\\
u_z=-w\  &\text{on}\  \partial\boz,\\
u_z=0\  &\text{on}\  \partial B(0,R).\\
\end{array}\r.
\end{equation}
Then, from Lemma \ref{l2.4}, we infer that the problem \eqref{e2.7} has a unique solution $u_z\in W^{1,p}(\boz_R)$.
Extend $u_z$ by zero on $\rn\backslash B(0,R)$. Then $u_z\in \dot{W}^{1,2}(\boz)\cap\dot{W}^{1,p}(\boz)$.
Let $v$ satisfy
\begin{equation}\label{e2.8}
\lf\{\begin{array}{ll}
-\mathrm{div}(A\nabla v)=\mathrm{div}(A\nabla u_z)\ \ &\text{in}\ \boz,\\
v=0\  &\text{on}\  \partial\boz.
\end{array}\r.
\end{equation}
By $u_z\in\dot{W}^{1,p}(\boz)$ and $u_z\equiv0$ on $\rn\backslash B(0,R)$, we conclude that
$\mathrm{div}(A\nabla u_z)\in \dot{W}^{-1,p}_0(\boz)$ has compact support.
From this and Lemma \ref{l2.3}, it follows that the problem \eqref{e2.8}
has a unique solution $v\in \dot{W}_0^{1,2}(\boz)\cap\dot{W}_0^{1,p}(\boz)$.
Thus, the problem \eqref{e2.6} has a unique solution $z=u_z+v\in\dot{W}^{1,2}(\boz)\cap\dot{W}^{1,p}(\boz)$.
Then $u=w+z\in\dot{W}_0^{1,p}(\boz)$ is a solution of the problem \eqref{e2.5}.

Meanwhile, by the definition of the space $\tilde{\mathcal{A}}^p_0(\boz)$, we find that the problem \eqref{e2.5}
has a unique solution in $\dot{W}_0^{1,p}(\boz)/\tilde{\mathcal{A}}^p_0(\boz)$.
Furthermore, a duality argument shows that
\begin{equation*}
\|f\|_{\dot{W}_0^{-1,p}(\boz)}=\|\mathrm{div}(A\nabla u)\|_{\dot{W}_0^{-1,p}(\boz)}
\ls\|u\|_{\dot{W}_0^{1,p}(\boz)/\tilde{\mathcal{A}}^p_0(\boz)},
\end{equation*}
which, together with the Open Mapping Theorem of Banach (see, for instance, \cite[Theorem 5.6-2]{c13}),
further implies that
$$
\|u\|_{\dot{W}_0^{1,p}(\boz)/\tilde{\mathcal{A}}^p_0(\boz)}\ls\|f\|_{\dot{W}_0^{-1,p}(\boz)},
$$
namely
$$\inf_{\phi\in\tilde{\mathcal{A}}^p_0(\boz)}\lf\|\nabla u-\nabla\phi\r\|_{L^p(\boz)}\ls\|f\|_{\dot{W}_0^{-1,p}(\boz)}.$$
This finishes the proof of Lemma \ref{l2.5}.
\end{proof}

Now, we prove Theorem \ref{t1.1} by using Lemma \ref{l2.5} and {Theorem \ref{app-dirichlet}(i)}.

\begin{proof}[Proof of Theorem \ref{t1.1}]
Let $2<p\in [n, \min\{q(\mathcal{L}),q(\mathcal{L}_{D,\boz_R})\})$. By {Theorem \ref{app-dirichlet}(i)} together with a duality argument,
we see that, for any given $q\in(1,\fz)$ and any $g\in L^q(\boz)$,
\begin{equation*}
\lf\|\mathcal{L}_D^{1/2}g\r\|_{\dot{W}_0^{-1,q}(\boz)}\ls\|g\|_{L^q(\boz)}.
\end{equation*}
From this and Lemma \ref{l2.5}, we infer that
\begin{equation*}
\inf_{\phi\in\tilde{\mathcal{A}}^p_0(\boz)}\lf\|\nabla f-\nabla\phi\r\|_{L^p(\boz)}\ls\lf\|\mathcal{L}_D f\r\|_{\dot{W}_0^{-1,p}(\boz)}
=\lf\|\mathcal{L}_D^{1/2}\mathcal{L}_D^{1/2}f\r\|_{\dot{W}_0^{-1,p}(\boz)}\ls
\lf\|\mathcal{L}_D^{1/2}f\r\|_{L^p(\boz)}.
\end{equation*}
This finishes the proof of Theorem \ref{t1.1}.
\end{proof}

\section{On the kernel space and completion of the proof}
In this section, we first identify $\mathcal{A}^p_0(\boz)$ with $\tilde{\mathcal{A}}^p_0(\boz)$, and then with  $\mathcal{K}_p(\mathcal{L}_D^{1/2})$,  and
finally complete the proof of Theorem \ref{c1.2}.

\begin{lemma}\label{l2.9}
Let $n=2$, $\boz\subset\rr^2$ be an exterior Lipschitz domain, and $p\in(2,q(\mathcal{L}))$. Then the problem
\begin{equation}\label{e3.1}
\mathcal{L}u=\frac{1}{|\partial\boz|}\mathbf{1}_{\partial\boz}\ \text{in}\ \rr^2
\end{equation}
has a unique solution $u\in\dot{W}^{1,p}(\rr^2)$ up to constants.
\end{lemma}

\begin{proof}
Let $q\in(2,\infty)$ and $t\in(1,\fz)$ be given by $\frac{1}{t}=\frac{2}{q'}-1$.
Then, by the Sobolev trace embedding theorem (see, for instance, \cite[Section 2.4.2, Theorem 4.2]{j13}), we find that, for any $\fai\in C^\fz_{\rm c}(\rn)$,
\begin{align*}
\lf|\lf\langle \frac{1}{|\partial\boz|}\mathbf{1}_{\partial\boz},\fai\r\rangle\r|
&=\lf|\frac{1}{|\partial\boz|}\int_{\partial\boz}\fai(x)\,d\sigma(x)\r|
\le\frac{1}{|\partial\boz|^{1/t}}\|\fai\|_{L^t(\partial\boz)}\\
&\ls\frac{1}{|\partial\boz|^{1/t}}\|\fai\|_{W^{1,q'}(\boz^c)}
\ls\frac{1}{|\partial\boz|^{1/t}}\|\fai\|_{\dot{W}^{1,q'}(\rn)}.
\end{align*}
Therefore, $\frac{1}{|\partial\boz|}\mathbf{1}_{\partial\boz}\in \dot{W}^{-1,q}(\rn)$ with any given $q\in(2,\infty)$.

Let $p\in(2,q(\mathcal{L}))$. From Lemma \ref{l2.0} with $n=2$, we deduce that there exists $f\in L^p(\rn)$ such that $\mathrm{div}f=\frac{1}{|\partial\boz|}\mathbf{1}_{\partial\boz}$. By this and
the assumption $p\in(2,q(\mathcal{L}))$, we further conclude that there exists
$u\in\dot{W}^{1,p}(\rn)$ such that $\mathcal{L}u=\frac{1}{|\partial\boz|}\mathbf{1}_{\partial\boz}$.

Moreover, if there exist $u_1,u_2\in\dot{W}^{1,p}(\rn)$ satisfy
$\mathcal{L}u_1=\frac{1}{|\partial\boz|}\mathbf{1}_{\partial\boz}=\mathcal{L}u_2$.
Then $u_1-u_2\in\dot{W}^{1,p}(\rn)$ and $\mathcal{L}(u_1-u_2)=0$, which, together
with $p\in(q(\mathcal{L})',q(\mathcal{L}))$, further implies that $u_1-u_2=c$.
Thus, the problem \eqref{e3.1} has a unique solution $u\in\dot{W}^{1,p}(\rn)$ up to constants. This finishes the proof of Lemma \ref{l2.9}.
\end{proof}

The following  was essentially obtained in \cite{sh05a}.
\begin{lemma}\label{l2.8}
Let $n\ge2$ and $O\subset\rn$ be a bounded Lipschitz domain. If {$A\in\mathrm{VMO}(\rn)$}, then there exists a positive
constant $\epz_0$, depending only on $n$ and the Lipschitz constant of $O$, such that, for any given
$p\in(\frac{4+\epz_0}{3+\epz_0},4+\epz_0)$ when $n=2$ or $p\in(\frac{3+\epz_0}{2+\epz_0},3+\epz_0)$ when $n\ge3$ ,
$\nabla\mathcal{L}_{D,O}^{-1}\mathrm{div}$ is bounded on $L^p(O)$. In particular, if $\partial O\in C^1$,
it holds that $\epz_0=\fz$; that is, $\nabla\mathcal{L}_{D,O}^{-1}\mathrm{div}$ is bounded on $L^p(O)$ for any $p\in(1,\fz)$.
\end{lemma}

We now identify  $\mathcal{A}^p_0(\boz)$ with $\tilde{\mathcal{A}}^p_0(\boz)$.

\begin{proposition}\label{p2.1}
Let $n\ge2$, $\boz\subset\rn$ be an exterior Lipschitz domain, and $p\in(1,\fz)$.
Take a large $R\in(0,\fz)$ such that $\boz^c\subset B(0,R)$ and let $\boz_R:=\boz\cap B(0,R)$.
Assume that $\min\{q(\mathcal{L}),q(\mathcal{L}_{D,\boz_R})\}>n$ and $2<p\in [n,\min\{q(\mathcal{L}),q(\mathcal{L}_{D,\boz_R})\})$. When $n\ge3$,
$$\tilde{\mathcal{A}}^p_0(\boz)=\mathcal{A}^p_0(\boz)=\{c(u_0-1): \ c\in\rr\},
$$
where $u_0$ is the unique solution in $\dot{W}^{1,2}(\boz)\cap\dot{W}^{1,p}(\boz)$ of the problem
\begin{equation}\label{e2.9}
\lf\{\begin{array}{ll}
-\mathrm{div}(A\nabla u_0)=0\ \ &\text{in}\ \boz,\\
u_0=1\  &\text{on}\  \partial\boz.
\end{array}\r.
\end{equation}
When $n=2$,
$$\tilde{\mathcal{A}}^p_0(\boz)=\mathcal{A}^p_0(\boz)=\{c(u_0-u_1): \ c\in\rr\},
$$
where $u_1$ is a solution of the problem \eqref{e3.1} and $u_0$ is the
unique solution in $\dot{W}^{1,2}(\boz)\cap\dot{W}^{1,p}(\boz)$ of the problem
\begin{equation}\label{e2.9a}
\lf\{\begin{array}{ll}
-\mathrm{div}(A\nabla u_0)=0\ \ &\text{in}\ \boz,\\
u_0=u_1\  &\text{on}\  \partial\boz.
\end{array}\r.
\end{equation}
\end{proposition}

\begin{proof}
{ For $2<p\in [n,\min\{q(\mathcal{L}),q(\mathcal{L}_{D,\boz_R})\})$ and $\phi\in\tilde{\mathcal{A}}^p_0(\boz)$, extend $\phi$ by zero in $\boz^c$}. Then the extension
of $\phi$, still denoted by $\phi$, belongs to $\dot{W}^{1,p}(\rn)$ and satisfies that
\begin{equation*}
\mathrm{div}(A\nabla\phi)=0\ \text{in}\ \boz, \ \mathrm{div}(A\nabla\phi)=0\ \text{in}\ \boz^c,\ \text{and}\ \phi=0\ \text{on}\ \partial\boz.
\end{equation*}
Since $\phi\in\dot{W}^{1,p}(\rn)$, it follows that $\frac{\partial \phi}{\partial\boldsymbol{\nu}}\in W^{-1/p,p}(\partial\boz)$,
where $\frac{\partial \phi}{\partial\boldsymbol{\nu}}:=(A\nabla \phi)\cdot\boldsymbol{\nu}$
denotes the \emph{conormal derivative} of $\phi$ on $\partial\boz$, and $W^{-1/p,p}(\partial\boz)$ denotes the
dual space of $W^{1/p,p'}(\partial\boz)$.
Moreover, it is easy to show that $\mathrm{div}(A\nabla\phi)$, as a distribution in $\rn$, satisfies that, for any
$\fai\in C^\fz_{\rm c}(\rn)$,
\begin{equation}\label{e3.0}
\langle\mathrm{div}(A\nabla\phi),\fai\rangle=-\lf\langle\frac{\partial\phi}{\partial\boldsymbol{\nu}},\fai\r\rangle_{\partial\boz},
\end{equation}
where $\langle\cdot,\cdot\rangle_{\partial\boz}$ denotes the duality pairing between $W^{-1/p,p}(\partial\boz)$ and
$W^{1/p,p'}(\partial\boz)$. Furthermore, let $h$ denote the distribution defined by $\mathrm{div}(A\nabla\phi)$;
that is, for any $\fai\in C^\fz_{\rm c}(\rn)$,
$$
\langle h,\fai\rangle=\langle\mathrm{div}(A\nabla\phi),\fai\rangle=-\lf\langle\frac{\partial\phi}{\partial\boldsymbol{\nu}},
\fai\r\rangle_{\partial\boz},
$$
which, combined with the Sobolev trace embedding theorem (see, for instance, \cite[Section 2.5.4, Theorem 5.5]{j13}),
further implies that, for any $\fai\in C^\fz_{\rm c}(\rn)$,
\begin{align*}
|\langle h,\fai\rangle|&\le\lf\|\frac{\partial\phi}{\partial\boldsymbol{\nu}}\r\|_{W^{-1/p,p}(\partial\boz)}
\|\fai\|_{W^{1/p,p'}(\partial\boz)}\ls\lf\|\frac{\partial\phi}{\partial\boldsymbol{\nu}}\r\|_{W^{-1/p,p}(\partial\boz)}
\|\fai\|_{W^{1,p'}(\boz^c)}\\
&\ls\lf\|\frac{\partial\phi}{\partial\boldsymbol{\nu}}\r\|_{W^{-1/p,p}(\partial\boz)}
\|\fai\|_{\dot{W}^{1,p'}(\rn)}.
\end{align*}
By this, we conclude that $h\in\dot{W}^{-1,p}(\rn)$ and $h$ has a compact support.

When $n\ge3$, from Lemma \ref{l2.2}, it follows that
the problem, that $\mathrm{div}(A\nabla w)=h$ in $\rn$, has a unique solution in $\dot{W}^{1,2}(\rn)\cap
\dot{W}^{1,p}(\rn)$ up to constants. Therefore, $w-\phi\in\dot{W}^{1,p}(\rn)$ and $\mathrm{div}(A\nabla (w-\phi))=0$ in $\rn$.
By this and the assumption $2<p\in [n,\min\{q(\mathcal{L}),q(\mathcal{L}_{D,\boz_R})\})$, we find that $w-\phi=c$ with $c\in\rr$, which, together with
the fact that $\mathrm{div}(A\nabla\phi)=0$ in $\boz$, implies that the restriction of $w$ to $\boz$
is the unique solution in $\dot{W}^{1,2}(\boz)\cap\dot{W}^{1,p}(\boz)$ of the problem
that $\mathrm{div}(A\nabla w)=0$ in $\boz$ and $w=c$ on $\partial\boz$.
Thus, $w=cu_0$ with $u_0$ being the same as in \eqref{e2.9} and $\phi=c(u_0-1)$.
This shows $\tilde{\mathcal{A}}^p_0(\boz)\subset {\mathcal{A}}^p_0(\boz).$

When $n=2$, without the loss of generality, we can assume that $\langle h,1\rangle\neq0$. Otherwise, $\langle h,1\rangle=0$,
we see that $h\in \dot{W}^{-1,2}(\rn)\cap \dot{W}^{-1,p}(\rn)$, then  the same proof as in the case of $n\ge3$ yields that
$\phi=c(u_0-1)$ with $u_0$ obtained by \eqref{e2.9}. However, since $1\in \dot{W}^{1,2}(\rr^2)\cap
\dot{W}^{1,p}(\rr^2)$, the uniqueness then implies $u_0=1$ and $\phi=0$.

Suppose now $\langle h,1\rangle\neq0$.  Let $u_1\in\dot{W}^{1,p}(\rn)$ be a solution of the problem
\eqref{e3.1}. Then $u_1$ satisfies
$$
\mathrm{div}(A\nabla u_1)=0\ \text{in}\ \boz,\ \ \mathrm{div}(A\nabla u_1)=0\ \text{in}\ \boz^c, \ \ \text{and} \ \langle -\mathrm{div}(A\nabla u_1),1\rangle=1.
$$
Let $w\in\dot{W}^{1,2}(\rr^2)\cap\dot{W}^{1,p}(\rr^2)$ satisfies
\begin{equation}\label{e3.2}
\mathrm{div}(A\nabla w)=h+\langle h,1\rangle\mathrm{div}(A\nabla u_1)\ \text{in}\ \rr^2.
\end{equation}
Indeed, by $h\in\dot{W}^{-1,p}(\rr^2)$ and
$$-\mathrm{div}(A\nabla u_1)=\frac{1}{|\partial\boz|}\mathbf{1}_{\partial\boz}\in \dot{W}^{-1,p}(\rr^2),$$
we conclude that the right-hand side of \eqref{e3.2} belongs to $\dot{W}^{-1,p}(\rr^2)$. This, together with the fact that
both $h$ and $\frac{1}{|\partial\boz|}\mathbf{1}_{\partial\boz}$ have compact supports and Lemma \ref{l2.1}, further implies
that
$$h+\langle h,1\rangle\mathrm{div}(A\nabla u_1)\in\dot{W}^{-1,2}(\rr^2).$$
Moreover, it is easy to find that
$$
\langle[h+\langle h,1\rangle\mathrm{div}(A\nabla u_1)],1\rangle=0.
$$
Therefore, from this and Lemma \ref{l2.2}, we deduce that the problem \eqref{e3.2} has a unique solution
$w\in\dot{W}^{1,2}(\rr^2)\cap\dot{W}^{1,p}(\rr^2)$ up to constants. Then, by the fact that $w$ satisfies \eqref{e3.2}, we conclude that
$w-\langle h,1\rangle u_1-\phi\in\dot{W}^{1,p}(\rr^2)$ and
$$
\mathrm{div}(A\nabla [w-\langle h,1\rangle u_1-\phi])=0\ \text{in}\ \rr^2,
$$
where $\phi$ is as in \eqref{e3.0}, which, combined with $p\in(2,q(\mathcal{L}))$, implies that
\begin{equation}\label{e3.3}
\phi=w-\langle h,1\rangle u_1-c \ \text{in}\ \rr^2,
\end{equation}
where $c$ is a constant.

From the boundary condition that $\phi=0$ on $\partial\boz$ and \eqref{e3.3}, we deduce that $w=c+\langle h,1\rangle u_1$ on $\partial\boz$.
Then the restriction of $w$ to $\boz$ is the unique solution in $\dot{W}^{1,2}(\boz)\cap\dot{W}^{1,p}(\boz)$ of the problem
that $\mathrm{div}(A\nabla w)=0$ in $\boz$ and $w=c+\langle h,1\rangle u_1$ on $\partial\boz$.
Moreover, let $w_1\in\dot{W}^{1,2}(\boz)\cap\dot{W}^{1,p}(\boz)$ be the unique solution of the problem
\begin{equation*}
\lf\{\begin{array}{ll}
\mathrm{div}(A\nabla w_1)=0\ \ &\text{in}\ \boz,\\
w_1=\langle h,1\rangle u_1\  &\text{on}\  \partial\boz,
\end{array}\r.
\end{equation*}
that is, $w_1=\langle h,1\rangle u_0$. Then, we find that $w=w_1+c=\langle h,1\rangle u_0+c.$ This, combined with \eqref{e3.3}, further conclude that
$$\phi=\langle h,1\rangle (u_0-u_1).$$
This shows $\tilde{\mathcal{A}}^p_0(\boz)\subset {\mathcal{A}}^p_0(\boz)$ for $n=2$.

The converse inclusion ${\mathcal{A}}^p_0(\boz)\subset \tilde{\mathcal{A}}^p_0(\boz)$ is obvious, since constants belongs to $\dot{W}^{1,p}(\rn)$ for $p\ge n$
and $u_1\in \dot{W}^{1,p}(\boz)$ for $n=2$.
\end{proof}

\begin{lemma}\label{l2.7}
Let $n\ge2$ and $p\in(1,\fz)$. Suppose that $A\in\mathrm{VMO}(\rn)$ satisfies $(GD)$, or $A\in\mathrm{CMO}(\rn)$.
Then the operator $\nabla\mathcal{L}^{-1}\mathrm{div}$ is bounded on $L^p(\rn)$.
\end{lemma}
\begin{proof}
The  case $A\in\mathrm{CMO}(\rn)$  follows from \cite[Theorem 1]{is98}.
For the case $A\in\mathrm{VMO}(\rn)$ satisfying $(GD)$, it follows from \cite{jl20} (see also \cite[Theorem 5.1 \& Proposition 5.2]{jl22}) that
$\nabla \mathcal{L}^{-1/2}$ is bounded on $L^p(\rn)$ for $1<p<\infty$. Thus we have
$${\|\nabla \mathcal{L}^{-1}\mathrm{div}\|_{p\to p}<\infty},$$
which gives the desired conclusion.
\end{proof}

We also point out that when both the matrix $A$ and $\boz$ have nice smoothness,
the function $u$ as in \eqref{e3.1} can be represented by using the fundamental solution associated with $\mathcal{L}$ (see, for instance, \cite[Theorem 2.7]{agg97}).
\begin{proposition}\label{p3.2}
Let $n\ge2$, $\boz\subset\rn$ be an exterior Lipschitz domain, and $p\in(1,\fz)$.
\begin{itemize}
\item[{\rm (i)}] If $\Omega$ is $C^1$, and $A\in\mathrm{CMO}(\rn)$, or {$A\in \mathrm{VMO}(\rn)$ satisfies $(GD)$}, then,
when $n\ge3$, for any $p\in[n, \infty)$,
$\tilde{\mathcal{A}}^p_0(\boz)=\{c(u_0-1): \ c\in\rr\}$ with $u_0$ being the same as in \eqref{e2.9}; when $n=2$, for any $p\in(2, \infty)$,
$\tilde{\mathcal{A}}^p_0(\boz)=\{c(u_0-u_1): \ c\in\rr\}$ with $u_0$ being the same as in \eqref{e2.9} and $u_1$ being a solution of the problem \eqref{e3.1}.

\item[{\rm (ii)}] Assume $\mathcal{L}_D:=\Delta_D$ and $\Omega$ is $C^{1,1}$. If $n\ge3$ and $p\in[n,\fz)$, then $\tilde{\mathcal{A}}^p_0(\boz)
=\{c\phi_\ast:\ c\in\rr\}$, where $\phi_\ast$ is the unique solution of the Dirichlet problem
\begin{equation*}
\lf\{\begin{array}{ll}
\Delta \phi_\ast=0\ \ &\text{in}\ \boz,\\
\phi_\ast=0\  &\text{on}\  \partial\boz,\\
\phi_\ast(x)\to1\  &\text{as}\ |x|\to\fz.\\
\end{array}\r.
\end{equation*}
If $n=2$ and $p\in(2,\fz)$, then $\tilde{\mathcal{A}}^p_0(\boz)=\{c\phi_\ast:\ c\in\rr\}$,
where $\phi_\ast$ is a harmonic function in $\boz$ satisfying that $\phi_\ast=0$ on $\partial\boz$ and
\begin{equation*}
\lf\{\begin{array}{ll}
\phi_\ast(x)=-c_0\ln|x|+O(|x|^{-1}),\\
\nabla\phi_\ast(x)=-c_0\nabla\ln|x|+O(|x|^{-2}),\\
\nabla^2\phi_\ast(x)=O(|x|^{-2}),\\
\end{array}\r.
\end{equation*}
as $|x|\to\fz$. Here $c_0$ is a constant and the notation $O(|x|^{-2})$ means that $\lim_{|x|\to\fz}\frac{|x|^{-2}}{O(|x|^{-2})}$
exists and is finite.
\end{itemize}
\end{proposition}

\begin{proof}
If $A\in\mathrm{CMO}(\rn)$, or { $A\in \mathrm{VMO}(\rn)$ satisfies $(GD)$}, from Lemma \ref{l2.7}, we infer that
$q(\mathcal{L})=\fz$. Moreover, by Lemma \ref{l2.8}, it holds that $ q(\mathcal{L}_{D,\boz_R})=\infty.$
Therefore,   by Proposition \ref{p2.1} we find that (i) holds.

The conclusion of (ii) was obtained in \cite[Theorem 2.7 and Remark 2.8]{agg97}
(see also \cite[Remarks 5.3, 5.4 and 5.5]{sw23}) and we omit the details here.
This finishes the proof of Proposition \ref{p3.2}.
\end{proof}

We prove Theorem \ref{c1.2} by using Theorem \ref{t1.1} and Proposition \ref{p3.2}.

\begin{proof}[Proof of {Theorem \ref{c1.2}}]

(i) Assume that $A\in\mathrm{VMO}(\rn)$ satisfies $(GD)$, or $A\in \mathrm{CMO}(\rn)$. Let $2<p\in [n,\infty)$.
Let us show that $\mathcal{K}_p(\mathcal{L}_D^{1/2})$ coincides with $\tilde{\mathcal{A}}^p_0(\boz)$.

Take a large constant $R\in(0,\fz)$ such that
$\boz^c\subset B(0,R-1)$ and let $\boz_R:=\boz\cap B(0,R)$. Then $\boz_R$ is a bounded $C^1$ domain of $\rn$.
By Lemma \ref{l2.7}, we find that $q(\mathcal{L})=\fz$. Moreover, from Lemma \ref{l2.8}, we infer that
$q(\mathcal{L}_{D,\boz_R})=\infty$. Therefore, by Theorem \ref{t1.1} and Proposition \ref{p3.2},
it holds for any $f\in \dot{W}_0^{1,p}(\boz)$,
\begin{equation*}
\inf_{\phi\in\tilde{\mathcal{A}}^p_0(\boz)}\lf\|\nabla f-\nabla\phi\r\|_{L^p(\boz)}\le C\lf\|\mathcal{L}^{1/2}_D f\r\|_{L^p(\boz)}.
\end{equation*}
This implies that $\mathcal{K}_p(\mathcal{L}_D^{1/2})\subset \tilde{\mathcal{A}}^p_0(\boz)$.

Let us show the converse inclusion. Let $u\in \mathcal{A}^p_0(\boz)$. By Theorem \ref{app-dirichlet} (i) we see that
$\mathcal{L}_D^{1/2}u\in L^p(\Omega)$.
Denote by $\{p^D_t\}_{t>0}$ the heat kernels of the heat semigroup
$\{e^{-t\mathcal{L}_D}\}_{t>0}$. By \cite[Lemma 2.3]{cd99}, there exists $\gamma>0$ such that for all $t>0$,
 \begin{eqnarray*}
\int_{\Omega}\lf|\nabla_x p^D_t(x,y)\r|^{2}\exp\left\{\gz |x-y|^2/t\right\}\,dx\le Ct^{1-\frac n2},
 \end{eqnarray*}
which implies for $1<q<2$ that
\begin{eqnarray*}
\int_{\Omega}\lf|\nabla_x p^D_t(x,y)\r|^{q}\exp\left\{\gz |x-y|^2/(2t)\right\}\,dx\le Ct^{\frac q2-\frac n2}.
\end{eqnarray*}
 Thus $p_t^D(x,\cdot)\in {W}_0^{1,q}(\Omega)$ for $1\le q\le 2$, for all $t>0$.
Therefore, for each $t>0$, $\mathcal{L}_De^{-t\mathcal{L}_D}u$ satisfies that for all $x\in\boz$,
\begin{align*}
\mathcal{L}_De^{-t\mathcal{L}_D}u(x)&=\int_\Omega(\mathcal{L}_D)_xp_t^D(x,y)u(y)\,dy=\int_\Omega (\mathcal{L}_D)_yp_t^D(x,y)u(y)\,dy\\
&=-\int_\Omega A(y)\nabla_yp_t^D(x,y)\cdot\nabla u(y)\,dy=\int_\Omega p_t^D(x,y)\mathcal{L}_D u(y)\,dy=0,
\end{align*}
where the second equality  by symmetry of the heat kernel, the third equality by $u\in \dot{W}_0^{1,p}(\Omega)$
and $p_t^D(x,\cdot)\in {W}_0^{1,p'}(\Omega)$, $1/p+1/p'=1$, $p>n$ when $n=2$ and $p\ge n$ when $n\ge 3$. We thus see that
$$
\mathcal{L}_D^{1/2}u=\frac{1}{\sqrt{\pi}} \int_{0}^{\infty}\mathcal{L}_D e^{-s \mathcal{L}_D} u\frac{\,ds}{\sqrt{s}}
=\frac{1}{\sqrt{\pi}} \int_{0}^{\infty} e^{-s \mathcal{L}_D} \mathcal{L}_D u\frac{\,ds}{\sqrt{s}}=0,
$$
which implies that $\tilde{\mathcal{A}}^p_0(\boz)\subset \mathcal{K}_p(\mathcal{L}_D^{1/2})$.

(ii) By (i) and Theorem \ref{t1.1}, we see that
for all
$f\in \dot{W}_0^{1,p}(\boz)$ that
\begin{equation*}
\inf_{\phi\in \mathcal{K}_p(\mathcal{L}_D^{1/2})}\lf\|\nabla f-\nabla\phi\r\|_{L^p(\boz)}\le C\lf\|\mathcal{L}^{1/2}_D f\r\|_{L^p(\boz)}.
\end{equation*}
This further implies that for all $f\in \dot{W}_0^{1,p}(\boz)$,
\begin{equation*}
\inf_{\phi\in \mathcal{K}_p(\mathcal{L}_D^{1/2})}\lf\|\nabla f-\nabla\phi\r\|_{L^p(\boz)}\le C\lf\|\mathcal{L}^{1/2}_D f\r\|_{L^p(\boz)}=\inf_{\phi\in \mathcal{K}_p(\mathcal{L}_D^{1/2})} C\lf\|\mathcal{L}^{1/2}_D (f-\phi)\r\|_{L^p(\boz)}.
\end{equation*}
By Theorem \ref{app-dirichlet}(i), it holds  for all
$f\in \dot{W}_0^{1,p}(\boz)$ and $\phi\in \mathcal{K}_p(\mathcal{L}_D^{1/2})$ that
\begin{equation*}
\lf\|\mathcal{L}^{1/2}_D f\r\|_{L^p(\boz)}=\lf\|\mathcal{L}^{1/2}_D (f-\phi)\r\|_{L^p(\boz)}\le
C\lf\|\nabla (f-\phi)\r\|_{L^p(\boz)}.
\end{equation*}
The last two inequalities give the desired conclusion and complete the proof.
\end{proof}

\smallskip

\noindent{\bf Acknowledgements.}\quad
The second author would like to thank Professor Zhongwei Shen
for some helpful discussions on the topic of this paper.

\bigskip

\noindent Renjin Jiang

\medskip

\noindent Academy for Multidisciplinary Studies, Capital Normal University, Beijing 100048,
People's Republic of China

\smallskip

\noindent {\it E-mail}: \texttt{rejiang@cnu.edu.cn} (R. Jiang)

\bigskip

\noindent Sibei Yang
\medskip

\noindent School of Mathematics and Statistics, Gansu Key Laboratory of Applied Mathematics
and Complex Systems, Lanzhou University, Lanzhou 730000, People's Republic of China

\smallskip

\smallskip

\noindent {\it E-mail}: \texttt{yangsb@lzu.edu.cn} (S. Yang)

\hspace{0.888cm}

\end{document}